\documentclass[11pt,twoside]{amsart}
\usepackage{amsmath,amssymb,amscd,amsthm,verbatim,alltt,amsfonts,array,curves,mathabx}
\usepackage{mathrsfs}
\usepackage[english]{babel}
\usepackage{latexsym}
\usepackage{amssymb}
\usepackage{euscript}
\usepackage{graphicx}

\numberwithin{equation}{section}
\newtheorem{theorem}{Theorem}[section]
\newtheorem{lemma}[theorem]{Lemma}

\theoremstyle{definition}
\newtheorem{definition}[theorem]{Definition} 
 
\newtheorem{remark}[theorem]{Remark}
\newtheorem{example}[theorem]{Example}

\begin{document}

%%%%%%%%%%%%%%%%%%%%%%%%%%%%%%%%%%%%%%%%%%%%%%%%%%%%%%%%%%%%%%%%%%%%%

\title[Relative Generalized Minimum Distance Functions ]{Relative Generalized Minimum Distance Functions } 

\thanks{The first and fourth authors were supported by COFAA-IPN and SNI,
Mexico. The third author was supported by SNI, Mexico. }
%\thanks{*Corresponding author}

%\author{Yuriko Pitones}
%\address{
%Departamento de
%Matem\'aticas\\
%Centro de Investigaci\'on y de Estudios
%Avanzados del
%IPN\\
%Apartado Postal
%14--740 \\
%07000 Mexico City, D.F.
%}
%\email{ypitones@math.cinvestav.mx}

%\author{Carlos Renter\'\i a-M\'arquez}
%\address{
%Departamento de Matem\'aticas\\
%Escuela Superior de F\'\i sica y
%Matem\'aticas\\
%Instituto Polit\'ecnico Nacional\\
%07300 Mexico City, D.F.}
%\email{renteri@esfm.ipn.mx}

\author[M. Gonz\'alez-Sarabia]{Manuel Gonz\'alez-Sarabia} 
\address{M.G. Sarabia. Instituto Polit\'ecnico Nacional, 
UPIITA, Av. IPN No. 2580,
Col. La Laguna Ticom\'an,
Gustavo A. Madero C.P. 07340,
 Ciudad de M\'exico. 
Departamento de Ciencias B\'asicas}
\email{mgonzalezsa@ipn.mx}

\author[M. E.Uribe--Paczka]{M. Eduardo Uribe--Paczka}
\address{
Instituto Polit\'ecnico Nacional\\
Escuela Superior de F\'{\i}sica y Matem\'aticas \\
Departamento de Matem\'aticas \\
07300, Ciudad de M\'exico,  M\'exico \\
}
\email{muribep1700@alumno.ipn.mx}

\author[E. Sarmiento]{Eliseo Sarmiento}
\address{
Instituto Polit\'ecnico Nacional\\
Escuela Superior de F\'{\i}sica y Matem\'aticas \\
Departamento de Matem\'aticas \\
07300, Ciudad de M\'exico,  M\'exico \\
}
\email{esarmiento@ipn.mx}

\author[C. Renter\'\i a]{Carlos Renter\'\i a}
\address{
Instituto Polit\'ecnico Nacional\\
Escuela Superior de F\'{\i}sica y Matem\'aticas \\
Departamento de Matem\'aticas \\
07300, Ciudad de M\'exico,  M\'exico \\
}
\email{renteri@esfm.ipn.mx}

\begin{abstract}
In this paper we introduce the relative generalized minimum distance function (RGMDF for short) and it allows us to give an algebraic approach to the relative generalized Hamming weights of the projective Reed--Muller--type codes. Also we introduce the relative generalized footprint function and it gives a tight lower bound for the RGMDF which is much easier to compute.	
\end{abstract}

\maketitle 

This work is a non--trivial generalization of \cite{hilbert-min-dis}, where the case of an algebraic approach to the minimum distance of a Reed--Muller--type code is treated, and \cite{rth-footprint}, where a similar approach is given for the case of the generalized Hamming weights of these codes. The main goal here is the study of the relative generalized Hamming weights (Definition \ref{RGHW1}) of the Reed--Muller--type codes from an algebraic point of view. In order to do this, we introduce the relative generalized minimum distance function (Definition \ref{RGMDF}) and the relative footprint function (Definition \ref{RGFF}).

The Reed--Muller--type codes and their parameters have been studied extensively. If $\mathbb{X}$ is a subset of a projective space $\mathbb{P}^{s-1}$ over a finite field $K=\mathbb{F}_q$, and $C_{\mathbb{X}}(d)$ is the corresponding Reed--Muller--type code (Definition \ref{RMTC}), several cases have been described \cite{carvalho-lopez-lopez}, \cite{duursma},\cite {GRT},\cite{sarabia1},\cite{sarabia2},\cite{sarabia3},\cite{sarabia5},\cite{sarabia4}, \-\cite{GHW2014},\cite{camps-sarabia-sarmiento-vila},\cite{lachaud}, \cite{hiram},\ \cite{neves}, \
\cite{vaz}, \cite{Veronese},\cite{IdealApproach},\cite{algcodes},\cite{ci-codes},\cite{sorensen}:
\begin{itemize}
\item Projective Reed--Muller codes: $\mathbb{X}=\mathbb{P}^{s-1}$.
\item Generalized Reed--Muller codes: $\mathbb{X}=\varphi (\mathbb{A}^{s-1})$, where $\mathbb{A}^{s-1}$ is an affine space and $\varphi: \mathbb{A}^{s-1} \rightarrow \mathbb{P}^{s-1}$, $\varphi(a_1,\ldots,a_{s-1})=[1:a_1: \cdots :a_{s-1}]$.
\item Reed--Muller--type codes arising from the Segre variety or the Veronese variety: $\mathbb{X}$ is the set of $K$--rational points of the variety.
\item Reed--Muller--type codes arising from a complete intersection: $\mathbb{X}$ is such that its defining ideal is a set--theoretic complete intersection.
\item Codes parameterized by a set of monomials: $\mathbb{X}$ is the toric set associated to these monomials.
\item Codes parameterized by the edges of a graph: $\mathbb{X}$ is the toric set associated to the edges of a simple graph.
\item Affine cartesian codes: $\mathbb{X}$ is the image of a cartesian product of subsets of $K$ under the map $K^{s-1} \rightarrow \mathbb{P}^{s-1}$, $x \rightarrow [x:1]$.
\item Projective cartesian codes: $\mathbb{X}$ is the image of the cartesian product $A_1 \times \cdots \times A_s \setminus \{\vec{0}\}$ under the map $K^s \setminus \{\vec{0}\} \rightarrow \mathbb{P}^{s-1}$, $x \rightarrow [x]$,
\end{itemize}
and others. 

On the other hand, the relative generalized Hamming weights (RGHW for short) of a linear code were introduced in \cite{Luo}. They are a natural generalization of the generalized Hamming weights introduced by Wei in \cite{wei}. The study of the RGHW is motivated because of their usefulness to protect messages from an adversary in the wire--tap channel of type II with illegitimate parties. Some properties of the RGHW of $q$--ary codes are described in \cite{Liu} and they are computed in the cases of almost all $4$--dimensional linear codes and their subcodes. Furthermore, some equivalences, inequalities and bounds are given in \cite{Zhuang}. The behavior of the RGHW of one point algebraic geometric codes is analyzed in \cite{Geil}. In the case of Hermitian codes, the RGHW are often much larger than the corres\-pon\-ding generalized Hamming weights. Also some bounds for the RGHW of some codes parameterized by a set of monomials of the same degree are given in \cite{sarabia6}. Particularly, the case of the codes parameterized by the edges of a connected bipartite graph is developed. Recently, in \cite{Geil2}, the authors use the footprint bound from  Gr\"obner basis theory to establish the true values of all corresponding RGHW for $q$--ary Reed--Muller codes in two variables. For the case of more variables they describe a simple and low complexity algorithm to determine the parameters.

The contents of this paper are as follows. In section \ref{prel} we introduce some concepts that will be needed throughout the paper. Particularly the definition of the relative generalized minimum distance function, which coincides with the relative generalized Hamming weights of certain Reed--Muller--type codes, and the definition of the relative generalized footprint function, which is a lower bound, easier to compute, for these weights.

In section \ref{MainResults} we show our main results. Theorems \ref{equivalence1} and \ref{Vasconcelos} give two algebraic equivalences for the relative generalized Hamming weights of some Reed--Muller--type codes: the relative generalized minimum distance function and the relative Vasconcelos function.
Also we prove that in the case of the relative generalized minimum distance function it is not necessary to analyze all the homogeneous polynomials of degree $d$. It is enough to study the standard polynomials (Theorem \ref{equivalence2}).
Finally, in Theorem \ref{LowerBound}, we show a lower bound for the relative generalized Hamming weights of some Reed--Muller--type codes which is easier to compute than the relative generalized minimum distance function

For additional information about Gr\"obner bases and Commutative Algebra, we refer to \cite{CLO,Ene-Herzog,monalg-rev}. For basic Coding Theory, we refer to \cite{MacWilliams-Sloane}.

\section{Preliminaries} \label{prel}

Let $S=K[t_1,\ldots,t_s]=\oplus_{d=0}^{\infty} S_d$ be a polynomial ring over a field $K$ with the standard grading. Let $I \neq (0)$ be a graded ideal of $S$ of Krull dimension $\beta$, and let $I_d=I \cap S_d$. The {\it{Hilbert function}} of $S/I$ is given by
\begin{align*}
& \hspace{0.7cm} H_I: \mathbb{N}_0 \rightarrow \mathbb{N}_0, \\
& H_I(d)=\dim_K(S_d/I_d),
\end{align*}
where $\mathbb{N}_0$ stands for the non--negative integers. It is known that there is a unique polynomial $h_I(x)=a_{\beta-1}x^{\beta-1}+\cdots+a_1x+a_0 \in \mathbb{Q}[x]$, $a_{\beta-1} \neq 0$, such that $h_I(d)=H_I(d)$ for $d \gg 0$. The {\it{degree or multiplicity}} of $S/I$ is the positive integer given by
$$
\deg(S/I)=\left\{
\begin{array}{lll}
(\beta-1)! \cdot a_{\beta-1} & {\mbox{if}} & \beta \geq 1, \\
& & \\
\dim_K(S/I) & {\mbox{if}} & \beta=0.
\end{array} \right.
$$

Particularly, in this work we consider mainly the case of finite fields, and if $\mathbb{X}$ is a subset of a projective space, we use as the graded ideal $I$ the vanishing ideal $I_{\mathbb{X}}$. In this situation, $\beta=1$ and the Hilbert polynomial is $|\mathbb{X}|$. Therefore $\deg(S/I)=|\mathbb{X}|$. However, the following definitions are valid for any field $K$ and any graded ideal $I \neq (0)$ of $S_d$.

Let $k_1 \in \ldbrack 0,k \rdbrack$, where $k=H_I(d)$ and $\ldbrack a,b \rdbrack:=\{x \in \mathbb{Z}: a \leq x \leq b\}$. If $k_1=0$ we define $\mathcal{W}_{d,k_1}=\{\emptyset\}$. If $k_1 \geq 1$ then let $\mathcal{W}_{d,k_1}$ be the set of all subsets $\{g_1,\ldots,g_{k_1}\}$ of $S_d$ such that $g_1+I,\ldots,g_{k_1}+I$ are linearly independent over $K$. Given $d \in \mathbb{N}$, $k_1 \in \ldbrack 0,k \rdbrack$, $r \in \ldbrack 1,k-k_1 \rdbrack$, $G \in \mathcal{W}_{d,k_1}$, we set
$$
U_{d,r,k_1,G} :=  \{\{f_1,\ldots,f_r\} \subseteq S_d :  \,  \{f_1,\ldots,f_r\} \cup G
\,  \, {\mbox{is $K$--linearly independent modulo $I$}}
\}.
$$
Also, we define
$$
\mathcal{F}_{d,r,k_1,G} :=  \{\{f_1,\ldots,f_r\} \in U_{d,r,k_1,G} : \,  (I:(f_1,\ldots,f_r)) \neq I
\},
$$
where $(I:(f_1,\ldots,f_r))=\{f \in S : f f_i \in I \,\, {\mbox{for all}} \,\, i\}$ is an ideal quotient.
We observe that if $k_1=0$ then $G=\emptyset$ and $\mathcal{F}_{d,r,k_1,G}$ is the set $\mathcal{F}_{d,r}$ introduced in \cite{rth-footprint}.

\begin{definition} \label{RGMDF}
{\it{The relative generalized minimum distance function}} (RGMDF for short) of I is the function $\delta_I: \mathbb{N} \times \ldbrack 1,k-k_1 \rdbrack \times \ldbrack 0,k \rdbrack \times \mathcal{W}_{d,k_1} \rightarrow \mathbb{Z}$ given by
$$
\delta_I(d,r,k_1,G)= 
\left\{ \begin{array}{lll}
\deg (S/I)-\max\{\deg (S/(I,F)): F \in \mathcal{F}_{d,r,k_1,G}\} & {\mbox{if}} & \mathcal{F}_{d,r,k_1,G} \neq \emptyset, \\
\deg (S/I) & {\mbox{if}} &  \mathcal{F}_{d,r,k_1,G} = \emptyset.
\end{array} \right.
$$
\end{definition}

We notice that if $k_1=0$ then $\delta_I(d,r,k_1,G)$ is equal to the generalized minimum distance function $\delta_I(d,r)$ that was introduced in \cite{rth-footprint}. Moreover, if $k_1=0$ and $r=1$ then $\delta_I(d,r,k_1,G)$ is equal to the minimum distance function $\delta_I(d)$, that was studied in \cite{hilbert-min-dis}.

On the other hand, let $\prec$ be a monomial order on $S$ and let $I$ be a non--zero ideal. If $f \in S, f \neq 0$, then $f=c_1t^{a_1}+ \cdots + c_mt^{a_m}$ with $c_i \in K \setminus \{0\}$ for all $i$, $t^{a_i}=t_1^{a_{i1}}\cdots t_s^{a_{is}}$, and $t^{a_1} \succ \cdots \succ t^{a_m}$. We recall that the {\it{leading monomial}} of $f$ is $t^{a_1}$ and it is denoted by ${\rm{in}}_{\prec}(f)$. The initial ideal of $I$ is the monomial ideal
$$
{\rm{in}}_{\prec}(I)=(\{{\rm{in}}_{\prec}(f): f \in I, f \neq 0\}).
$$

%Also we recall that the set $\{t^{a_1},\ldots,t^{a_m}\}$ is the support of the polynomial $f$.

\begin{definition} \label{FP}
The {\it{footprint}} of $S/I$, denoted $\Delta_{\prec}(I)$, is the set of all the monomials that are not the leading monomial of any polynomial in $I$. The elements of the footprint of $S/I$ are called {\it{standard monomials}}. A polynomial $f$ is called {\it{standard}} if $f \neq 0$ and $f$ is a $K$--linear combination of standard monomials.
\end{definition}

Actually, if 
\begin{align*}
\pi: S \rightarrow S/I, \\
\pi(x)=x+I,
\end{align*}
then $\pi(\Delta_{\prec}(I))$ is a basis of $S/I$ as a $K$--vector space, and the image of the standard polynomials of degree $d$ is $S_d/I_d$. Hence, if $I$ is a graded ideal, $|\Delta_{\prec}(I) \cap S_d|=H_I(d)$.

Furthermore, if $\prec$ is a monomial order on $S$ and $\Delta_{\prec} (I)_d:=\Delta_{\prec} (I) \cap S_d$,  then we set
\begin{align*}
 \mu_{\prec,d,r,k_1,G}:=  & \{\{t^{a_1},\ldots,t^{a_r}\} \subset \Delta_{\prec}(I)_{d}: t^{a_1},\ldots,t^{a_r},{\rm{in}}_{\prec}(g_1),\ldots,{\rm{in}}_{\prec}(g_{k_1}) \,\, {\mbox{are distinct }} \\
& {\mbox{ monomials}},
 \, \, {\mbox{and}} \, ({\rm{in}}_{\prec} \, (I):(M)) \neq {\rm{in}}_{\prec} \, (I)\}.
\end{align*}
Notice that, for the goal of this work, there is no loss of generality if we consider that $${\rm{in}}_{\prec}(g_1),\ldots,{\rm{in}}_{\prec}(g_{k_1})$$ are distinct monomials for any $G \in \mathcal{W}_{d,k_1}$ (see the induction process in the proof of \cite[Proposition 4.8]{rth-footprint} and the codes (\ref{code1})).
\begin{definition} \label{RGFF}
 {\it{The relative generalized footprint function}} (RGFF for short) of $I$ is the function ${\rm{fp}}_I: \mathbb{N} \times \ldbrack 1,k-k_1 \rdbrack \times \ldbrack 0,k \rdbrack \times \mathcal{W}_{d,k_1} \rightarrow \mathbb{Z}$ given by
 \begin{align*}
& {\rm{fp}}_I(d,r,k_1,G)= \\
& \hspace{1.2cm} \left\{ \begin{array}{lll}
\deg (S/I)-\max\{\deg (S/({\rm{in}}_{\prec}(I),M)): M \in \mu_{\prec,d,r,k_1,G}\} & {\mbox{if}} & \mu_{\prec,d,r,k_1,G} \neq \emptyset,\\
\deg (S/I) & {\mbox{if}} &  \mu_{\prec,d,r,k_1,G} = \emptyset.
\end{array} \right.
\end{align*}
\end{definition}

We observe that if $k_1=0$ then ${\rm{fp}}_I(d,r,k_1,G)$ is equal to the generalized footprint function ${\rm{fp}}_I(d,r)$ that was introduced in \cite{rth-footprint}. Moreover, if $k_1=0$ and $r=1$ then ${\rm{fp}}_I(d,r,k_1,G)$ is equal to the footprint function ${\rm{fp}}_I(d)$, that was studied in \cite{hilbert-min-dis}.
 Now, to relate these concepts with the relative generalized Hamming weights of certain linear codes, we need to recall this definition. Let $C$ be an $[s,k]$ linear code, that is, $C$ is a linear subspace of $K^s$, where $K$ is a finite field with $q$ elements,  $\dim C=k$, and let $C_1$ be a subspace of $C$ with $\dim C_1=k_1$. 
\begin{definition} \label{RGHW1}
{\it{The $r$th relative generalized Hamming weight}} of $C$ and $C_1$ is given by
$$
M_r(C,C_1)=\min \{\, |{\mbox{supp}} \, (D)|: \, D \, \, {\mbox{is a subspace of}} \, \,C, 
\dim (D)=r, \, D\cap C_1=\{\vec{0}\}\},
$$
for all $r=1,\ldots,k-k_1$. 
\end{definition}
Particularly if $r=1$ we realize that
$$
M_1(C,C_1)=\min \{w({\bf{x}}): {\bf{x}} \in C \setminus C_1\}.
$$
where $w({\bf{x}})$ is the Hamming weight of ${\bf{x}}$ (the number of non--zero entries of ${\bf{x}}$).
In the case that $C_1=\{\vec{0}\}$, we obtain the $r$th generalized Hamming weight of $C$,
$$
\delta_r(C)=\min \{\, |{\mbox{supp}} \, (D)|: D \, {\mbox{is a subspace of}} \, C, \, \dim (D)=r \}.
$$

That is, $\delta_r(C)=M_r(C,\{\vec{0}\})$ for all $r=1,\ldots,k$.
Moreover, the linear codes where these concepts match are the projective Reed--Muller-type codes. We recall their definition. Let $K=\mathbb{F}_q$ be a finite field with $q$ elements, let $\mathbb{P}^{s-1}$ be a projective space over $K$ and let $\mathbb{X}=\{P_1,\dots,P_m\}$ be a subset of $\mathbb{P}^{s-1}$. We assume that the points of $\mathbb{X}$ are in standard position, that is, the first non--zero entry is $1$.
\begin{definition} \label{RMTC}
{\it{The projective Reed--Muller--type code}} of degree $d$ on $\mathbb{X}$ is the image of the following evaluation map:
\begin{align*}
&  \hspace{0.8cm} {\rm{ev}} \, : S_d \rightarrow K^m, \\
& f \rightarrow (f(P_1),\ldots,f(P_m)),
\end{align*}
and it is denoted by $C_{\mathbb{X}}(d)$. The vanishing ideal of $\mathbb{X}$, denoted $I_{\mathbb{X}}$, is the ideal of $S$ generated by the homogeneous polynomials that vanish at all points of $\mathbb{X}$.
\end{definition}
From now on we will use the following notation: if $f \in S_d$ then $\Lambda_f:=(f(P_1),\ldots,f(P_m)) \in C_{\mathbb{X}}(d)$, that is, $\Lambda_f={\rm{ev}} \, (f)$.
Furthermore, if $G \in \mathcal{W}_{d,k_1}$, we set
\begin{equation} \label{code1}
C_{\mathbb{X}}(d,k_1,G):= 
 \{ \Lambda_{g}  \in K^m : g \in \langle G \rangle \},
\end{equation}
where $\langle G \rangle$ is the subspace of $S_d$ generated by $G$.
Notice that $C_{\mathbb{X}}(d,k_1,G)$ is a subspace of $C_{\mathbb{X}}(d)$. Actually, if $k_1=0$ then $G=\emptyset$ and $C_{\mathbb{X}}(d,k_1,G)=\{\vec{0}\}$. The main goal of this paper is to show that $M_r(C_{\mathbb{X}}(d),C_{\mathbb{X}}(d,k_1,G))=\delta_{I_{\mathbb{X}}}(d,r,k_1,G) \geq {\rm{fp}}_{I_{\mathbb{X}}}(d,r,k_1,G)$ for all $d \geq 1$, $0 \leq k_1 \leq k$, $1 \leq r \leq k-k_1$, and $G \in \mathcal{W}_{d,k_1}$. It gives us an efficient lower bound for the relative generalized Hamming weights of the Reed--Muller--type codes that is much easier to compute than the RGMDF.

\section{Main results} \label{MainResults}

\begin{lemma} \label{li}
Let $\mathbb{X} \subseteq \mathbb{P}^{s-1}$ and $I=I_{\mathbb{X}}$ its vanishing ideal. Let $h_1,\ldots,h_l \in S_d, l \leq k$. Then $h_1+I,\ldots,h_l+I$ are linearly independent over $K$ if and only if $\Lambda_{h_1},\ldots,\Lambda_{h_l}$ are linearly independent vectors of $C_{\mathbb{X}}(d)$.
\end{lemma}

\begin{proof}
$\Rightarrow$) Suppose that $h_1+I,\ldots,h_{k_1}+I$ are linearly independent over $K$. If $\sum_{i=1}^{l} a_i \Lambda_{h_i}=\vec{0}$ for some $a_i \in K$, then $\sum_{i=1}^{l} a_ih_i \in I$. Thus $a_i=0$ for all $i=1,\ldots,l$, and the claim follows.

$\Leftarrow )$ If $\Lambda_1,\ldots,\Lambda_{h_l}$ are linearly independent vectors of $C_{\mathbb{X}}(d)$ and $\sum_{i=1}^l b_i (h_i+I)=I$ for some $b_i \in K$, then $\sum_{i=1}^l b_i h_i \in I$. Therefore
 $$\sum_{i=1}^l b_i \Lambda_{h_i}=\Lambda_{\sum_{i=1}^l b_ih_i}=\vec{0}.$$ 
 
 Then $b_i=0$ for all $i=1,\ldots,l$, and the implication follows. 
\end{proof} 

\begin{remark}
Lemma \ref{li} proves that $\{\Lambda_{g_1},\ldots,\Lambda_{g_{k_1}}\}$ is a basis of $C_{\mathbb{X}}(d,k_1,G)$ when $k_1 \geq 1$. Therefore, $\dim_K C_{\mathbb{X}}(d,k_1,G)=k_1$ for all $k_1 \geq 0$.
\end{remark}

\begin{lemma} \label{eq1}
%\begin{enumerate}
If $D$ is a subspace of $C_{\mathbb{X}}(d)$ with $\dim_K D=r$, $d \geq 1$, $1 \leq r \leq H_{I}(d)-k_1$, and $\{\Lambda_{f_1},\ldots,\Lambda_{f_r}\}$ is a basis of $D$, then $D \cap C_{\mathbb{X}}(d,k_1,G)=\{\vec{0}\}$ if and only if $\{f_1,\ldots,f_r\} \in U_{d,r,k_1,G}$.
%\end{enumerate}
\end{lemma}

\begin{proof}
If $k_1=0$ the claim follows immediately. Let $k_1 \geq 1$. 

$\Rightarrow$) Suppose that $D \cap C_{\mathbb{X}}(d,k_1,G)=\{\vec{0}\}$. If $\{f_1,\ldots,f_r \} \notin U_{d,r,k_1,G}$ then $\{f_1,\ldots,f_r\} \cup G$ is linearly dependent modulo $I$. Thus there are $a_1,\ldots,a_r,b_1,\ldots,b_{k_1} \in K$, not all of them equal to zero, such that
$$
\sum_{i=1}^r a_if_i+\sum_{i=1}^{k_1} b_ig_i \in I.
$$  

Let $f=\sum_{i=1}^r a_if_i$, $g=\sum_{i=1}^{k_1} b_ig_i$. Thus $\Lambda_{f+g}=\vec{0}=\Lambda_f+\Lambda_g$. Hence $\Lambda_f=-\Lambda_g$. Therefore $\Lambda_f \in D \cap C_{\mathbb{X}}(d,k_1,G)$. If $\Lambda_f=\vec{0}$ then $f \in I$, $g \in I$ and $a_i=0$ for all $i=1,\ldots,r$, $b_i=0$ for all $i=1,\ldots,k_1$, a contradiction. Then $\Lambda_f \neq \vec{0}$ and this contradicts that $D \cap C_{\mathbb{X}}(d,k_1,G)=\{\vec{0}\}$

$\Leftarrow$) Suppose that  $\{f_1,\ldots,f_r\} \in U_{d,r,k_1,G}$. Then $\{f_1,\ldots,f_r \} \cup G$ is linearly independent modulo I. If $\Lambda_g \in D \cap C_{\mathbb{X}}(d,k_1,G)$ then
$$
\Lambda_g=\sum_{i=1}^r a_i \Lambda_{f_i}=\sum_{i=1}^{k_1} b_i \Lambda_{g_i}=\Lambda_{\sum_{i=1}^r a_i f_i}=\Lambda_{\sum_{i=1}^{k_1} b_i g_i}
$$
for some $a_i, b_i \in K$. Hence $\sum_{i=1}^r a_if_i-\sum_{i=1}^{k_1} b_ig_i \in I$. Therefore $\sum_{i=1}^r a_i(f_i+I)-\sum_{i=1}^{k_1} b_i (g_i+I)=I$. But $f_1+I,\ldots,f_r+I,g_1+I,\ldots,g_{k_1}+I$ are linearly independent over $K$. Thus $a_i=0$ for all $i=1,\ldots,r$, and $b_i=0$ for all $i=1,\ldots,k_1$. Then $\Lambda_g=\vec{0}$, and the claim follows.
\end{proof}

In the next Lemma we use the following notation: if $F=\{f_1,\ldots,f_r\} \subseteq S_d$, then the set of zeros of $F$ in $\mathbb{X}$ is given by
$$
V_{\mathbb{X}}(F)=\{[P] \in \mathbb{X}: f_i(P)=0 \, \, {\mbox{for all}}\, \, i=1,\ldots,r\}.
$$
\begin{lemma} \label{Relative}
If $\mathbb{X} \subset \mathbb{P}^{s-1}$, $d \geq 1$, $1 \leq r \leq H_{I}(d)-k_1$, $G \in \mathcal{W}_{d,k_1}$, then
$$
M_r(C_{\mathbb{X}}(d),C_{\mathbb{X}}(d,k_1,G))   =  \min \{|\mathbb{X} \setminus V_{\mathbb{X}}(F)|:  
 F=\{f_1,\ldots,f_r\} \in U_{d,r,k_1,G}
\}.
$$
\end{lemma}

\begin{proof}
If $D$ is a subspace of $C_{\mathbb{X}}(d)$ with $\dim_K D=r$, and $\{\Lambda_{f_1},\ldots,\Lambda_{f_r}\}$ is a $K$--basis of $D$ with $F=\{f_1,\ldots,f_r\} \subseteq S_d$, then,
by \cite[Lemma 4.3]{rth-footprint}, we know that
\begin{equation} \label{support}
|\mbox{supp} \, (D)|=|\mathbb{X}-V_{\mathbb{X}}(F)|.
\end{equation}

The claim follows at once from (\ref{support}), Lemma \ref{eq1}, and the definition of the $r$th relative generalized Hamming weight $M_r(C_{\mathbb{X}}(d),C_{\mathbb{X}}(d,k_1,G))$. 
\end{proof}

The following theorem gives an algebraic approach to the relative generalized Hamming weights of the Reed--Muller--type codes.

\begin{theorem} \label{equivalence1}
Let $K$ be a finite field, $\mathbb{X} \subseteq \mathbb{P}^{s-1}$, and $I=I_{\mathbb{X}}$ its vanishing ideal. Let $G \in \mathcal{W}_{d,k_1}$. Then
$$
M_r(C_{\mathbb{X}}(d),C_{\mathbb{X}}(d,k_1,G))  = \delta_{I} (d,r,k_1,G),
$$
for all $d \geq 1$, $0 \leq k_1 \leq H_I(d)$, and $1 \leq r \leq H_I(d)-k_1$.
\end{theorem}

\begin{proof}
If $k_1=0$ then $C_{\mathbb{X}}(d,k_1,G)=\{\vec{0}\}$ and $M_r(C_{\mathbb{X}}(d),C_{\mathbb{X}}(d,k_1,G))=\delta_r(C_{\mathbb{X}}(d))$. Also $\delta_I(d,r,k_1,G)$ is the generalized minimum distance $\delta_I(d,r)$. Therefore the claim follows from \cite[Theorem 4.5]{rth-footprint}. Let $k_1 \geq 1$. If $\mathcal{F}_{d,r,k_1,G}=\emptyset$ then $\delta_I(d,r,k_1,G)=\deg (S/I)=|\mathbb{X}|$. Moreover if $F=\{f_i\}_{i=1}^r \subseteq U_{d,r,k_1,G}$ then $(I:(F))=I$. By \cite[Lemma 3.2]{rth-footprint} and Lemma \ref{Relative} we obtain that $M_r(C_{\mathbb{X}}(d),C_{\mathbb{X}}(d,k_1,G))=\deg (S/I)=|\mathbb{X}|$, and the equality follows. Assume that $\mathcal{F}_{d,r,k_1,G} \neq \emptyset$. Using Lemma \ref{Relative}, \cite[Lemma 3.4]{rth-footprint} and the fact that $\deg (S/I)=|\mathbb{X}|$ we obtain that
\begin{align*}
M_r(C_{\mathbb{X}}(d),C_{\mathbb{X}}(d,k_1,G))  & = \min \{|\mathbb{X} \setminus V_{\mathbb{X}}(F)| : F \in \mathcal{F}_{d,r,k_1,G}\}\\
& = |\mathbb{X}| - \max \{|V_{\mathbb{X}}(F) | : F \in \mathcal{F}_{d,r,k_1,G}\} \\
& =  |\mathbb{X}| - \max \{\deg (S/(I,F)) : F \in \mathcal{F}_{d,r,k_1,G}\} \\
& = \deg (S/I) - \max \{\deg (S/(I,F)) : F \in \mathcal{F}_{d,r,k_1,G}\} \\
& = \delta_I (d,r,k_1,G), 
\end{align*}
and the result follows. 
\end{proof}

\begin{definition}
Let $I$ be a graded ideal of $S$ and $G \in \mathcal{W}_{d,k_1}$. The {\it{relative Vasconcelos function}} of $I$ is the function $\vartheta_I: \mathbb{N} \times \ldbrack 1,k-k_1 \rdbrack \times \ldbrack 0,k \rdbrack \times \mathcal{W}_{d,k_1} \rightarrow \mathbb{N}$ given by
$$
\vartheta_I(d,r,k_1,G)=\left\{ \begin{array}{lll}
\min \{\deg (S/(I:(F))): F \in \mathcal{F}_{d,r,k_1,G}\} & {\mbox{if}} & \mathcal{F}_{d,r,k_1,G} \neq \emptyset, \\
\deg (S/I) & {\mbox{if}} &  \mathcal{F}_{d,r,k_1,G} = \emptyset.
\end{array} \right.
$$
\end{definition}

We notice that if $k_1=0$ then the relative Vasconcelos function is the Vasconcelos function $\vartheta_I(d,r)$, introduced in \cite[Definition 4.4]{rth-footprint}.

\begin{theorem} \label{Vasconcelos}
Let $K$ be a finite field, $\mathbb{X} \subseteq \mathbb{P}^{s-1}$, and $I=I_{\mathbb{X}}$ its vanishing ideal. Let $G \in \mathcal{W}_{d,k_1}$. Then
$$
M_r(C_{\mathbb{X}}(d),C_{\mathbb{X}}(d,k_1,G))  = \vartheta_{I} (d,r,k_1,G),
$$
for all $d \geq 1$, $0 \leq k_1 \leq H_I(d)$, and $1 \leq r \leq H_I(d)-k_1$.
\end{theorem}

\begin{proof}
If $k_1=0$ then $M_r(C_{\mathbb{X}}(d),C_{\mathbb{X}}(d,k_1,G))  = \delta_r(C_{\mathbb{X}}(d))$ and the relative Vasconcelos function is the Vasconcelos function $\vartheta_I(d,r)$. The claim follows from \cite[Theorem 4.5]{rth-footprint}. Let $k_1 \geq 1$. If $\mathcal{F}_{d,r,k_1,G}=\emptyset$ then 
$$
\vartheta_I(d,r,k_1,G)=\deg(S/I)=|\mathbb{X}|=M_r(C_{\mathbb{X}}(d),C_{\mathbb{X}}(d,k_1,G)),
$$
as was observed in the proof of Theorem \ref{equivalence1}. Assume $\mathcal{F}_{d,r,k_1,G} \neq \emptyset$. Using Lemma \ref{Relative} and \cite[Lemma 3.2]{rth-footprint} we get
\begin{align*}
M_r(C_{\mathbb{X}}(d),C_{\mathbb{X}}(d,k_1,G))  & = \min \{|\mathbb{X} \setminus V_{\mathbb{X}}(F)| : F \in \mathcal{F}_{d,r,k_1,G}\}\\
& = \min \{\deg(S/(I:(F)): F \in \mathcal{F}_{d,r,k_1,G}\} \\
& = \vartheta_I(d,r,k_1,G),
\end{align*}
and the claim follows. 
\end{proof}

\begin{lemma} \label{standard}
Let $F=\{f_1,\ldots,f_r\} \subseteq S_d$ be a set of standard polynomials such that the leading monomials ${\rm{in}}_{\prec}(f_1),\ldots,{\rm{in}}_{\prec}(f_r)$ are distinct. Therefore $f_1+I,\ldots,f_r+I$ are linearly independent over $K$.
\end{lemma}

\begin{proof}
If $h:=\sum_{i=1}^r a_if_i \in I$ for some $a_i \in K$, and $a_j \neq 0$ for some $j=1,\ldots,r$, then $h \neq 0$ and ${\rm{in}}_{\prec}(h) \in \Delta_{\prec}(I)$, a contradiction. Thus $a_i=0$ for all $i=1,\ldots,r$ and $f_1+I,\ldots,f_r+I$ are linearly independent over $K$. 
\end{proof}

Let $\mathcal{F}_{\prec,d,r,k_1,G}$ be the set of all subsets $F=\{f_1,\ldots,f_r\} \in U_{d,r,k_1,G}$ such that $(I:(F)) \neq I$, $f_i$ is a standard polynomial for all $i=1,\ldots,r$, and 
$${\rm{in}}_{\prec}(f_1),\ldots,{\rm{in}}_{\prec}(f_r),{\rm{in}}_{\prec}(g_1),\ldots,{\rm{in}}_{\prec}(g_{k_1})$$
 are distinct monomials. The following theorem allows us to work just with the standard polynomials instead of all the polynomials to study the RGMDF of $I$.

\begin{theorem} \label{equivalence2}
Let $d \in \mathbb{N}$, $r \in \ldbrack 1, k-k_1 \rdbrack$, $k_1 \in \ldbrack 0,k \rdbrack$, and $G \in \mathcal{W}_{d,k_1}$. The RGMDF of $I$ is given by
$$
\delta_I(d,r,k_1,G)= 
 \left\{ \begin{array}{lll}
\deg (S/I)-\max\{\deg (S/(I,F)): F \in \mathcal{F}_{\prec,d,r,k_1,G}\} & {\mbox{if}} & \mathcal{F}_{\prec,d,r,k_1,G} \neq \emptyset, \\
\deg (S/I) & {\mbox{if}} &  \mathcal{F}_{\prec,d,r,k_1,G} = \emptyset.
\end{array} \right.
$$
\end{theorem}

\begin{proof}
If $k_1=0$ the result follows from \cite[Proposition 4.8]{rth-footprint}. We assume $k_1 \geq 1$. Take $F=\{f_1,\ldots,f_r\} \in \mathcal{F}_{d,r,k_1,G}$. By the proof of \cite[Proposition 4.8]{rth-footprint}, $f_i=p_i+h_i$ with $p_i \in I_d$ and $h_i$ is a $K$--linear combination of standard monomials of degree $d$. Setting $H=\{h_1,\ldots,h_r\}$, we observe that $(I:(F))=(I,(H))$, $(I,F)=(I,H)$, and $f_i+I=h_i+I$ for $i=1,\ldots,r$. We need to show that $H \in U_{d,r,k_1,G}$. If $H \notin U_{d,r,k_1,G}$ then $H \cup G$ is linearly dependent modulo I. But then $F \cup G$ is linearly dependent modulo I, a contradiction because $F \in \mathcal{F}_{d,r,k_1,G}$. Hence, $H \in U_{d,r,k_1,G}$. Therefore we may replace $F$ by $H$, and, with this assumption, as in the same proof of \cite[Proposition 4.8]{rth-footprint}, there is a set $G_1=\{\mathcal{G}_1,\ldots,\mathcal{G}_r\}$ of homogeneous standard polynomials of degree $d$ such that $\langle F \rangle=\langle G_1 \rangle$, ${\rm{in}}_{\prec}(\mathcal{G}_1), \ldots, {\rm{in}}_{\prec}(\mathcal{G}_r)$ are distinct monomials and ${\rm{in}}_{\prec}(f_i) \succeq {\rm{in}}_{\prec}(\mathcal{G}_i)$ for all $i$. Analogously, we can assume that 
$${\rm{in}}_{\prec}(h_1),\ldots,{\rm{in}}_{\prec}(h_r),{\rm{in}}_{\prec}(g_1),\ldots,{\rm{in}}_{\prec}(g_{k_1})$$
are distinct monomials.
It remains to prove that $G_1 \in U_{d,r,k_1,G}$. On the contrary, if $G_1 \notin U_{d,r,k_1,G}$ then $G_1 \cup G$ is linearly dependent modulo I. Thus
$$
\sum_{i=1}^r a_i \mathcal{G}_i+\sum_{i=1}^{k_1} b_i g_i \in I
$$
for some $a_i, b_i \in K$, and at least one of the $a_i \neq 0$ and one of the $b_i \neq 0$. But, as $\langle F \rangle=\langle G_1 \rangle$,
$$
\sum_{i=1}^r a_i \mathcal{G}_i=\sum_{i=1}^r c_i f_i
$$
for some $c_i \in K$, not all of them equal to zero. Therefore $F \cup G$ is linearly dependent modulo $I$, a contradiction.
Hence $G_1 \in U_{d,r,k_1,G}$, and the claim follows. 
\end{proof}

\begin{remark} \label{number1}
Although Theorem \ref{equivalence2} gives an interesting algebraic equivalence for the RGMDF of I, it is hard to compute this number because as $H_I(d)=|\Delta_{\prec}(I) \cap S_d|$, the number of subsets of $r$ standard polynomials in $U_{d,r,k_1,G}$ is at most $\binom{q^{H_I(d)}-1}{r}$, and then we need to test which of them are in $\mathcal{F}_{\prec,d,r,k_1,G}$ and compute the corresponding degrees.
\end{remark}

%The following theorem shows an efficient lower bound for the RGMDF of $I$ which is much easier to compute.

\begin{theorem} \label{LowerBound}
Let $K$ be a finite field, $\mathbb{X} \subseteq \mathbb{P}^{s-1}$, $I=I_{\mathbb{X}}$ its vanishing ideal, and $G \in \mathcal{W}_{d,k_1}$. Then
$$
M_r(C_{\mathbb{X}}(d),C_{\mathbb{X}}(d,k_1,G))  \geq {\rm{fp}}_I(d,r,k_1,G),
$$
for all $d \geq 1$, $0 \leq k_1 \leq H_I(d)$, and $1 \leq r \leq H_I(d)-k_1$.
\end{theorem}

\begin{proof}
If $k_1=0$ then $M_r(C_{\mathbb{X}}(d),C_{\mathbb{X}}(d,k_1,G))=\delta_r(C_{\mathbb{X}}(d))$ and ${\rm{fp}}_I(d,r,k_1,G)$ is equal to the footprint function ${\rm{fp}}_I(d,r)$. The claim follows from \cite[Theorem 4.9]{rth-footprint}. Let $k_1 \geq 1$. If $\mathcal{F}_{\prec,d,r,k_1,G}=\emptyset$ then $\delta_I(d,r,k_1,G)=\deg(S/I)$, and by definition
$$
{\rm{fp}}_I(d,r,k_1,G) \leq \deg(S/I)=\delta_I(d,r,k_1,G)=M_r(C_{\mathbb{X}}(d),C_{\mathbb{X}}(d,k_1,G)).
$$

Assume $\mathcal{F}_{\prec,d,r,k_1,G} \neq \emptyset$, and let $F \in \mathcal{F}_{\prec,d,r,k_1,G}$. Thus $(I:(F)) \neq I$ and by \cite[Lemma 4.7]{rth-footprint}, $({\rm{in}}_{\prec}(I): ({\rm{in}}_{\prec} (F))) \neq {\rm{in}}_{\prec} (I)$, where ${\rm{in}}_{\prec} (F)=\{{\rm{in}}_{\prec} (f_1),\ldots,{\rm{in}}_{\prec} (f_r)\}$. 
Therefore ${\rm{in}}_{\prec}(F) \in \mu_{\prec,d,r,k_1,G}$, and, by \cite[Lemma 4.1]{rth-footprint},
$$
\deg(S/(I,F)) \leq  \deg(S/({\rm{in}}_{\prec}(I),{\rm{in}}_{\prec}(F))) 
 \leq  \max\{\deg (S/({\rm{in}}_{\prec}(I),M)): M \in \mu_{\prec,d,r,k_1,G}\}.
$$

Thus
$$
\max\{\deg (S/(I,F)): F \in \mathcal{F}_{d,r,k_1,G}\} \leq 
\max\{\deg (S/({\rm{in}}_{\prec}(I),M)): M \in \mu_{\prec,d,r,k_1,G}\},
$$
and then
\begin{align*}
&\deg(S/I)-\max\{\deg (S/(I,F)): \, F \in \mathcal{F}_{d,r,k_1,G}\} \geq \\
& \hspace{5cm}\deg(S/I)-\max\{\deg (S/({\rm{in}}_{\prec}(I),M)): M \in \mu_{\prec,d,r,k_1,G}\}.
\end{align*}

Hence $\delta_I(d,r,k_1,G) \geq {\rm{fp}}_{I}(d,r,k_1,G)$, and by Theorem \ref{equivalence1},
$$
M_r(C_{\mathbb{X}}(d),C_{\mathbb{X}}(d,k_1,G))  \geq {\rm{fp}}_I(d,r,k_1,G).
$$ 
\end{proof}

\begin{remark}
${\rm{fp}}_I(d,r,k_1,G)$ is easier to compute than $\delta_I(d,r,k_1,G)$ (and therefore than $M_r(C_{\mathbb{X}}(d),C_{\mathbb{X}}(d,k_1,G))$) because we need to test which of the at most $\binom{H_I(d)}{r}$ subsets of $r$ standard monomials are in $\mu_{\prec,d,r,k_1,G}$ and compute the corresponding degrees. And $\binom{H_I(d)}{r}$ is much lower than the value $\binom{q^{H_I(d)}-1}{r}$, given in Remark \ref{number1}.
\end{remark}

\section{Examples}

\begin{example} \label{ex1}
Let $K=\mathbb{F}_5$ be a finite field with $5$ elements, $S=K[t_1,t_2,t_3]$ be a polynomial ring , and  let $\mathbb{X}$ be a projective torus in $\mathbb{P}^{2}$, that is,
$$
\mathbb{X}=\mathbb{T}_{2}:=\{[z_1:z_2:z_3] \in \mathbb{P}^2 : z_i \in K \setminus \{0\}, \, {\mbox{for}} \, \, i=1,2,3\}.
$$

It is well kown that its vanishing ideal is given by
$$
I=I_{\mathbb{X}}=(t_1^4-t_3^4,t_2^4-t_3^4),
$$
and that ${\rm{reg}} (S/I)=6$, $\deg(S/I)=16$ (see for example \cite[Proposition 2.1]{ci-codes}). Actually, the Hilbert function is given in Table \ref{tab:1}.

\begin{table}
%\centering
\caption{The Hilbert function of $S/I$ in Example \ref{ex1}}
\label{tab:1}       % Give a unique label
\begin{tabular}{lllllll}
\hline\noalign{\smallskip}
$d$ & 1 & 2 & 3 & 4 & 5 & 6  \\
\noalign{\smallskip}\hline\noalign{\smallskip}
$H_I(d)$ & 3 & 6 & 10 & 13 & 15 & 16 \\
\noalign{\smallskip}\hline
\end{tabular}
\end{table}

Consider the case $k_1=0$. Thus $C_{\mathbb{X}}(d,k_1,G)=\{\vec{0}\}$ and $$M_r(C_{\mathbb{X}}(d),C_{\mathbb{X}}(d,k_1,G))=\delta_r(C_{\mathbb{X}}(d)).$$

Using Macaulay 2 \cite{mac2} we obtain the $6 \times 16$ matrix whose entry $(i,j)$ is precisely ${\rm{fp}}_I(i,j,k_1,G)$. That is, the number of the row is the value of $d$, and the number of the column is the value of $r$, and the entries are the values of the generalized footprint function:
$$
\left(
\begin{array}{cccccccccccccccc}
12 & 15 & 16 & - & - & - & - & - & - & - & - & - & - & - & - & - \\
8 & 11 & 12 & 14 & 15 & 16 & - & - & - & - & - & - & - & - & - & - \\
4 & 7 & 8 & 10 & 11 & 12 & 13 & 14 & 15 & 16 & - & - & - & - & - & - \\
3 & 4 & 6 & 7 & 8 & 9 & 10 & 11 & 12 & 13 & 14 & 15 & 16 & - & - & - \\
2 & 3 & 4 & 5 & 6 & 7 & 8 & 9 & 10 & 11 & 12 & 13 & 14 & 15 & 16 & -  \\
1 & 2 & 3 & 4 & 5 & 6 & 7 & 8 & 9 & 10 & 11 & 12 & 13 & 14 & 15 & 16
\end{array} \right)
$$

Using \cite[Theorem 18]{camps-sarabia-sarmiento-vila}, \cite[Corollary 2.3]{GHWVeronese}, \cite[Theorem 3.5]{ci-codes}, and Macaulay 2, we observe that the values of the generalized Hamming weights of $C_{\mathbb{X}}(d)$ are exactly the same that the entries of the last matrix. Therefore, for this particular example,
$$
{\rm{fp}}_I(d,r,k_1,G)=M_r(C_{\mathbb{X}}(d),C_{\mathbb{X}}(d,r,k_1))=\delta_r(C_{\mathbb{X}}(d)),
$$
for $k_1=0$, $d \in \ldbrack 1, 6 \rdbrack$, $r \in \ldbrack 1, H_{I}(d) \rdbrack$. Hence, the lower bound given in Theorem \ref{LowerBound} is attained.

\end{example}

\begin{example}
Let $K=\mathbb{F}_3$ be a finite field with $3$ elements, $S=K[t_1,t_2,t_3,t_4]$ be a polynomial ring with $4$ variables, and let $\mathbb{X}$ be a projective torus in $\mathbb{P}^3$. Thus
$$
\mathbb{X}=\mathbb{T}_3:=\{[z_1:z_2:z_3:z_4] \in \mathbb{P}^3 : z_i \in K^* \, {\mbox{for all}}  \, \, i\},
$$
where $K^*=K \setminus \{0\}$. The vanishing ideal of this set is given by
$$
I=I_{\mathbb{X}}=(t_1^2-t_4^2,t_2^2-t_4^2,t_3^2-t_4^2),
$$
and $\rm{reg} \, (S/I)=3$, $\deg(S/I)=8$. Assume $d=1$, $k_1=1$, $G=\{t_1\}$. As $H_I(1)=4$ then $1 \leq r \leq H_I(1)-k_1=3$. We notice that
$$
C_{\mathbb{X}}(1,1,\{t_1\})=\{(0,0,0,0,0,0,0,0),(1,1,1,1,1,1,1,1),(2,2,2,2,2,2,2,2)\}.
$$

\smallskip
{\bf{Case I:}} $r=1$. By \cite[Theorem 3.5]{ci-codes} we obtain that $\delta_1(C_{\mathbb{X}}(1))=4$. Also, using the generalized Plotkin bound \cite[Proposition 4]{Zhuang} we get
$$
4=\delta_1(C_{\mathbb{X}}(1)) \leq M_1(C_{\mathbb{X}}(1),C_{\mathbb{X}}(1,1,\{t_1\})) \leq \left\lfloor \frac{1-3^{-1}}{1-3^{-3}} \cdot (7) \right\rfloor=4. 
$$

Therefore, in this case,
$$
\delta_1(C_{\mathbb{X}}(1)) = M_1(C_{\mathbb{X}}(1),C_{\mathbb{X}}(1,1,\{t_1\})) =4.
$$

Furthermore, using Definition \ref{RGFF} and Macaulay 2 we obtain that
$$
{\rm{fp}}_I(1,1,1,\{t_1\})=4.
$$

\smallskip
{\bf{Case II:}} $r=2$. By \cite[Theorem 18]{camps-sarabia-sarmiento-vila} we obtain that $\delta_2(C_{\mathbb{X}}(1))=6$. Moreover, if we use the generalized Singleton bound \cite[Proposition 3]
{Zhuang}, we get that
$$
6=\delta_2(C_{\mathbb{X}}(1)) \leq M_2(C_{\mathbb{X}}(1),C_{\mathbb{X}}(1,1,\{t_1\})) \leq |\mathbb{X}|-H_I(1)+2=8-4+2=6. 
$$

Hence
$$
\delta_2(C_{\mathbb{X}}(1)) = M_2(C_{\mathbb{X}}(1),C_{\mathbb{X}}(1,1,\{t_1\}))=6.
$$

In the same way, using Definition \ref{RGFF} and Macaulay 2, we obtain that
$$
{\rm{fp}}_I(1,2,1,\{t_1\})=6.
$$

\smallskip
{\bf{Case III:}} $r=3$. By \cite[Corollary 2.3]{GHWVeronese} we obtain that $\delta_3(C_{\mathbb{X}}(1))=7$. Also, by the generalized Singleton bound,
$$
7=\delta_3(C_{\mathbb{X}}(1)) \leq M_3(C_{\mathbb{X}}(1),C_{\mathbb{X}}(1,1,\{t_1\})) \leq |\mathbb{X}|-H_I(1)+2=8-4+3=7. 
$$

Hence
$$
\delta_3(C_{\mathbb{X}}(1)) = M_3(C_{\mathbb{X}}(1),C_{\mathbb{X}}(1,1,\{t_1\}))=7.
$$

Using Definition \ref{RGFF} and Macaulay 2, we obtain that
$$
{\rm{fp}}_I(1,3,1,\{t_1\})=7.
$$

Therefore, in the three cases above, the lower bound of Theorem \ref{LowerBound} is attain\-ed.
\end{example}

\medskip

\noindent {\bf Acknowledgments.} The authors thank professor R. H. Villarreal because he provided some procedures in Macaulay 2. 

%reading of the paper and for the improvements suggested. 

\bibliographystyle{plain}

\end{document}